\documentclass[12pt]{amsart}
\usepackage[osf,sc]{mathpazo}
\usepackage{amssymb}

\usepackage{geometry}
\usepackage{bm}
\usepackage{graphicx}
\usepackage{tikz-cd}
\usepackage{hyperref}
\hypersetup{
	colorlinks=true, 
	linktoc=all,     
	linkcolor=blue,
	citecolor=red,
	filecolor=black,
	urlcolor=blue	
}
\usepackage{enumerate}
\usepackage[inline]{enumitem}
\makeatletter
\newcommand{\inlineitem}[1][]{%
	\ifnum\enit@type=\tw@
	{\descriptionlabel{#1}}
	\hspace{\labelsep}%
	\else
	\ifnum\enit@type=\z@
	\refstepcounter{\@listctr}\fi
	\quad\@itemlabel\hspace{\labelsep}%
	\fi} \makeatother
\newcommand{\Gs}{\Sigma}

\newcommand{\Gom}{\Omega}

\newcommand{\mbb}{\mathbb}

\newcommand{\mcl}{\mathcal}

\newcommand{\us}{\underset}

\newcommand{\lra}{\longrightarrow}
\newcommand{\llra}{\longleftrightarrow}

\newcommand{\Z}{\mbb Z}

\newcommand{\Ra}{\Rightarrow}

\newcommand{\equ}[1]{%
	\begin{equation*}
		#1
	\end{equation*}
}
\newcommand{\equa}[1]{%
	\begin{equation*}
		\begin{aligned}
			#1
		\end{aligned}
	\end{equation*}
}

\newcommand{\mattwofour}[8]{%
	\begin{pmatrix}
		#1 & #2 & #3 & #4\\#5 & #6 & #7 & #8
	\end{pmatrix}
}

\theoremstyle{plain}
\newtheorem{theorem}{Theorem}[section]

\newtheorem{lem}[theorem]{Lemma}

\newtheorem{conj}[theorem]{Conjecture}

\newtheorem{claim}[theorem]{Claim}

\newtheorem{note}[theorem]{Note}

\makeatletter
\def\namedlabel#1#2{\begingroup
	\def\@currentlabel{#2}%
	\label{#1}\endgroup
}
\makeatother

\newtheorem*{thmOmega}{\bf{Theorem} $\bm{\Gom}$}

\theoremstyle{definition}
\newtheorem{defn}[theorem]{Definition}

\theoremstyle{remark}
\newtheorem{remark}[theorem]{Remark}
\newtheorem{example}[theorem]{Example}
\numberwithin{equation}{section}

\begin{document}
\title[On the Coherent Labelling Conjecture of a Polyhedron]{On the Coherent Labelling Conjecture of a Polyhedron
in Three Dimensions}
\author[C.P. Anil Kumar]{C.P. Anil Kumar}
\address{Center for Study of Science, Technology and Policy
\# 18 \& \# 19, 10th Cross, Mayura Street,
Papanna Layout, Nagashettyhalli, RMV II Stage,
Bengaluru - 560094
Karnataka,INDIA
}
\address{Post Doctoral Fellow in Mathematics, Room No. 223, Middle Floor, Main Building, Harish-Chandra Research Institute, (Department of Atomic Energy, Government of India),
	Chhatnag Road, Jhunsi, Prayagraj (Allahabad)-211019, Uttar Pradesh, INDIA
}
\email{akcp1728@gmail.com}
\thanks{The author is supported by a research grant and facilities provided by Center for study of Science, Technology and Policy (CSTEP), Bengaluru, INDIA for this research work. The work is also partially done while the author is a Post Doctoral Fellow at Harish-Chandra Research Institute, Allahabad.}
\date{\sc \today}
\begin{abstract}
In this article we consider an open conjecture about coherently labelling a polyhedron in three dimensions.  We exhibit all the forty eight possible 
coherent labellings of a tetrahedron. We also exhibit that some simplicial polyhedra like bipyramids, Kleetopes, gyroelongated bipyramids are coherently labellable. Also we prove that pyramids over $n$-gons for $n\geq 4$, which are not simplicial polyhedra, are coherently labellable.  We prove that among platonic solids, the cube and the dodecahedron are not coherently labellable, even though, the tetrahedron, the octahedron and the icosahedron are coherently labellable.
Unlike the case of a tetrahedron, in general for a polyhedron, we show that a coherent labelling need not induce a coherent labelling at a vertex.
We prove the main conjecture in the affirmative for a certain class of polyhedra which are constructible from tetrahedra through certain types of edge and face vanishing tetrahedron attachments. As a consequence we conclude that a cube cannot be obtained from only these type of tetrahedron attachments. We also give an obstruction criterion for a polyhedron to be not coherently labellable and consequentially show that any polyhedron obtained from a pyramid with its apex chopped off is not coherently labellable. Finally with the suggestion of the affirmative results we prove the main theorem that any simplicial polyhedron is coherently labellable.
\end{abstract}
\maketitle 
\section{\bf{Introduction}}
For any convex polygon in a plane made up of $n$-edges, we know that, we can orient and label the edges successively 
with integers $1,2,\cdots, n$ with a choice of an anticlockwise successive pattern of inequalities
\equ{1<2<\cdots<n.}
This has an interesting generalization to a coherent labelling of a polyhedron in higher dimensions. Here we consider the conjecture of coherently labelling a polyhedron in three dimensions. Before we state the main Conjecture~\ref{conj:Labelling} we need a definition.
\begin{defn}[Coherent Labelling]
\label{defn:CL}
~\\
Let $\mcl{P}$ be a convex polyhedron in three dimensional space made of finitely many polygonal faces.
Choosing an outward normal for $\mcl{P}$ we orient each polygonal face $F$ and hence its edges with respect 
to the oriented face $F$ in an anticlockwise manner. We say a labelling of all the edges of $\mcl{P}$ with integers
is coherent if for each polygonal oriented face $F$ with $n_F$ edges, the labels have a choice of the following 
successive pattern of inequalities 
\equ{a_1<a_2<\cdots <a_{n_F}}
with respect to the orientation of $F$ in an anticlockwise manner. If such a labelling exists then the polyhedron $\mcl{P}$ is said to be coherently labellable. Otherwise it is said to be not coherently labellable.
\end{defn}
\begin{remark}
In all the observations made on a polyhedron in this article, the observer's eye is outside the polyhedron and not inside the polyhedron.	
\end{remark}
Now we state the open question on coherent labelling inequalities arising from 
a polyhedron in three dimensions.
\begin{conj}[Coherent Labelling Conjecture]
\label{conj:Labelling}
Classify those convex polyhedra in three dimensions which are coherently labellable and those which are not.
\end{conj}

In this article we make some progress regarding this conjecture. We first prove in Theorem~\ref{theorem:Tetrahedron} that there are forty eight coherent labellings of a tetrahedron and exhibit a coherent labelling for pyramids in Section~\ref{sec:CLIP}.
We prove in Theorem~\ref{theorem:PolyhedronConstruction} that a certain class of polyhedra are coherently labellable. This class of polyhedra is constructed using tetrahedra as building blocks with only certain type of edge and face vanishing attachments mentioned in Section~\ref{sec:TA}. This class includes pyramids (which include tetrahedra) and bipyramids (which include octahedra). The method of proof involves positioning tetrahedra one by one to construct the polyhedron with only allowed 
type of attachments. As a consequence we prove that all bipyramids are coherently labellable in Theorem~\ref{theorem:Bipyramids}.

In Section~\ref{sec:LPP} we prove Theorem~\ref{theorem:LPP} which says that any Kleetope (Definition~\ref{defn:PP}) is coherently labellable. We also prove Theorem~\ref{theorem:OnePyramid}, an extension theorem for coherently labelling a polyhedron, that is, if a polyhedron is coherently labellable then a coherent labelling exists for the polyhedron which is obtained by attaching a pyramid to any face of the given polyhedron such that the apriori edges and vertices of the face do not vanish, but the attached face vanishes.

We also see that it is not always possible to label any polyhedron
coherently. Some polyhedra which cannot be labelled coherently are cubes or cuboids (refer to 
Theorem~\ref{theorem:Cube}) and triangular prisms as a consequence of Theorem~\ref{theorem:ChopPyramid}. These are counterexamples. Actually we mention Theorem~\ref{theorem:ChopPyramid} where we prove that a certain class of polyhedra is not coherently labellable by a giving an obstruction criterion.

By proving that the dodecahedron is not coherently labellable and the icosahedron or any gyroelongated polyhedron is coherently labellable we complete all the cases of five regular polyhedra in Theorem~\ref{theorem:PlatonicSolids}. 

The general Conjecture~\ref{conj:Labelling} is still 
open. With the affirmative results of labellability of tetrahedron, bipyramids, Kleetopes and gyroelongated polyehedron, we prove in main Theorem~\ref{theorem:SP} that any simiplicial polyhedron is coherently labellable.

Now we mention another slightly different conjecture which is there in the literature regarding positioning of $d$-dimensional simplices in $\mbb{R}^d$ where a similar type of visualization is required as in Section~\ref{sec:TA} for various tetrahedron attachments.

\begin{conj}[Due to F.~Bagemihl, The Institute of Advanced Study]
~\\
The maximal number of pairwise touching $d$-simplices in a configuration in $\mbb{R}^d$ is \equ{f(d)=2^d.}
\end{conj}
This conjecture was first posed by F.~Bagemihl~\cite{MR0077132} in $1956$. This conjecture is also mentioned in the book
M.~Aigner and G.~M.~Ziegler~\cite{MR3823190} in the chapter on ``Touching Simplices". Also refer to H.~Tietze~\cite{MR0181558}, chapter IV on ``Neighbouring Domains", pp. 64-89. 


\section{\bf{Tetrahedron and its Forty Eight Coherent Labelling Inequalities}}
\label{sec:TCLI}
~\\
In this section we list all possible coherent labelling inequalities for a tetrahedron.
\begin{theorem}
\label{theorem:Tetrahedron}
Let $T$ be a tetrahedron with three pair-wise edge labelling symbols \equ{(x_1x_4),(x_2x_5),(x_3x_6)}
as in Figure~\ref{fig:One}. Then there exist forty eight coherent labellings of the tetrahedron given by
the cycle elements as 
\equa{(x_1x_2\cdots,x_6) \in \{&(124635),(134625),(125634),(135624),\\
&(135246),(136245),(145236),(146235)\}.}
\end{theorem}
\begin{figure}[h]
	\centering
	\includegraphics[width = 0.6\textwidth]{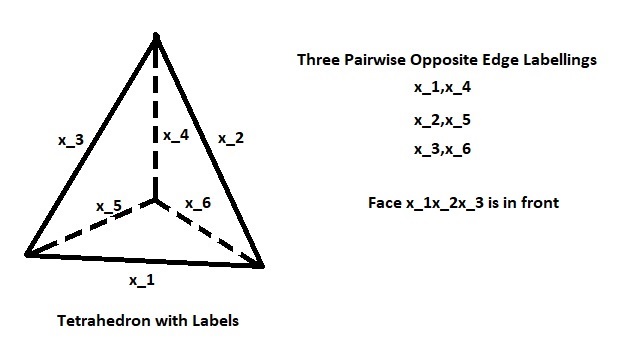}
	\caption{Tetrahedron $T$ with Labelling Symbols}
	\label{fig:One}
\end{figure}
\begin{proof}
There are six edges of a tetrahedron $T$. Hence there are totally $6!$ labellings. 
A cyclic change of labels given by
\equ{x_1\lra x_2\lra x_3\lra x_4\lra x_5\lra x_6 \lra x_1}
of a coherent labelling also gives rise to a coherent labelling. So we fix the least label for an edge.
Let $x_1=1$. Now the oriented faces corresponding to an outward pointing normal of $T$ 
from Figure~\ref{fig:One} are given by 
\equ{(x_1x_2x_3),(x_2x_6x_4),(x_3x_4x_5),(x_1x_5x_6)}
Since $x_1$ is the least we have 
\equ{x_1<x_2<x_3,x_1<x_5<x_6.} Now we have nine possibilities for the remaining
as given by 
\equa{&x_2<x_6<x_4,x_6<x_4<x_2,x_4<x_2<x_6\\
&x_3<x_4<x_5,x_4<x_5<x_3,x_5<x_3<x_4}
We will compute the possible consistent labellings.
\begin{enumerate}
\item $x_2<x_6<x_4,x_3<x_4<x_5 \Ra x_6<x_4<x_5<x_6$. Hence not possible.
\item $x_2<x_6<x_4,x_4<x_5<x_3 \Ra x_6<x_4<x_5<x_6$. Hence not possible.
\item $x_2<x_6<x_4,x_5<x_3<x_4 \Ra \mattwofour{1=x_1<}{x_2<}{x_3<}{}{}{x_5<}{x_6<}{x_4}$.

Every label in the left column if it exists has a comparison with every label in the right column.
Two labels in the same column do not have a definite comparison.
We have four possibilities. 
\equa{x_1=1<x_2=2<x_5=3<x_3=4<x_6=5<x_4=6 &\Ra (124635)\\
x_1=1<x_5=2<x_2=3<x_3=4<x_6=5<x_4=6 &\Ra (134625)\\
x_1=1<x_2=2<x_5=3<x_6=4<x_3=5<x_4=6 &\Ra (125634)\\
x_1=1<x_5=2<x_2=3<x_6=4<x_3=5<x_4=6 &\Ra (135624)\\
}
\item $x_6<x_4<x_2,x_3<x_4<x_5 \Ra x_6<x_4<x_5<x_6$. Hence not possible.
\item $x_6<x_4<x_2,x_4<x_5<x_3 \Ra x_6<x_4<x_5<x_6$. Hence not possible.
\item $x_6<x_4<x_2,x_5<x_3<x_4 \Ra x_4<x_2<x_3<x_4$. Hence not possible.
\item $x_4<x_2<x_6,x_3<x_4<x_5 \Ra x_4<x_2<x_3<x_4$. Hence not possible.
\item $x_4<x_2<x_6,x_4<x_5<x_3 \Ra \mattwofour{x_1=1<}{x_4<}{x_2<}{x_3}{}{}{x_5<}{x_6}$.

Every label in the left column if it exists has a comparison with every label in the right column.
Two labels in the same column do not have a definite comparison.
We have four possibilities. 
\equa{x_1=1<x_4=2<x_2=3<x_5=4<x_3=5<x_6=6 &\Ra (135246)\\
x_1=1<x_4=2<x_2=3<x_5=4<x_6=5<x_3=6 &\Ra (136245)\\
x_1=1<x_4=2<x_5=3<x_2=4<x_3=5<x_6=6 &\Ra (145236)\\
x_1=1<x_4=2<x_5=3<x_2=4<x_6=5<x_3=6 &\Ra (146235)}
\item $x_4<x_2<x_6,x_5<x_3<x_4 \Ra x_3<x_4<x_2<x_3$. Hence not possible.
\end{enumerate}
So the possible permutations are the following eight permutations with $x_1=1$.
\equ{\{(124635),(134625),(125634),(135624),(135246),(136245),(145236),(146235)\}.}
This proves the theorem.
\end{proof}
\subsection{Elementary Cyclic Group $(\Z/2\Z)^3$ Action on Labellings}

Consider the set of $48$ coherent labels of the tetrahedron defined as follows.
\equa{\mcl{L}_{T}&=\bigg\{(x_1,x_2,x_3,x_4,x_5,x_6)\in \{1,2,3,4,5,6\}^6\mid (x_1x_2x_3x_4x_5x_6)\in 
\{(124635),\\ &(134625),
(125634),(135624),(135246),(136245),(145236),(146235)\}\bigg\}}
Then we observe that we can interchange the labels of
\equ{x_1\llra x_4,x_2\llra x_5,x_3\llra x_6}
independently for a coherent labelling to get another coherent labelling.
So we have an action of $(\Z/2\Z)^3$ on the set $\mcl{L}_{T}$ via these interchanges.

\section{\bf{Coherent Labelling Inequalities of Pyramids}}
\label{sec:CLIP}
In this section we give a way to coherently label pyramids. 
We already have a coherent labelling of a tetrahedron. The labelling on pyramids is obtained by coherent choice of a set
of inequalities arising from every face when a tetrahedron is attached to a pyramid with base an $n$-gon to a pyramid 
with base an $(n+1)-$gon. We will observe this method of attachment of a tetrahedron again in Section~\ref{sec:TA}. 
The standard extension of inequalities given by 
\equ{1<2<3<\cdots<2n}
is a linear extension of all strings of inequalities arising out of the oriented faces of the pyramid.
A coherent labelling for a pyramid with base either a quadrilateral or a pentagon or a hexagon is given in Figure~\ref{fig:Two}.
\begin{figure}[h]
\centering
\includegraphics[width = 0.7\textwidth]{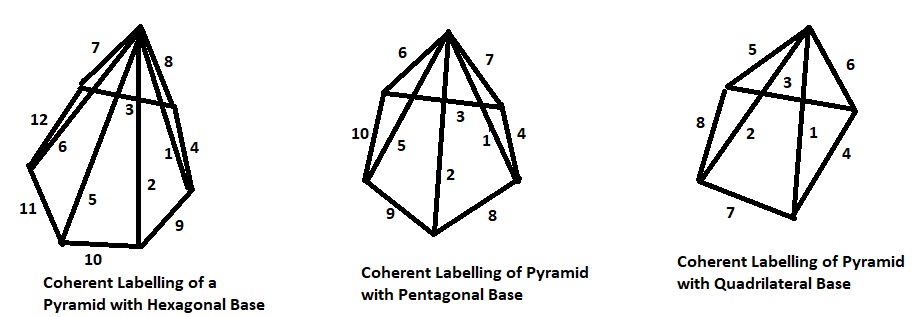}
\caption{Pyramid Labelling}
\label{fig:Two}
\end{figure}
In general a coherent labelling is obtained by extending this method of labelling as follows.
The base has labels 
\equ{3<4<(n+3)<(n+4)<\cdots<(2n)}
which gives an outward pointing normal for the pyramid.
The faces with apex has the following labels all of the them give outward pointing normal orientation.
Four of the faces with apex have the following labels.
\equ{3<(n+1)<(n+2),1<4<(n+2),1<2<(n+3),2<5<(n+4)}
and the remaining faces have the following labels for $n\geq 5$.
\equ{5<6<(n+5),6<7<(n+6),\cdots,(n-1)<n<(2n-1),n<(n+1)<2n}
This gives a coherent labelling of any pyramid.

\section{\bf{On Certain Types of Tetrahedron Attachments}}
\label{sec:TA}
In this section we attach a tetrahedron to an already labelled polyhedron to construct new polyhedron with a coherent 
labelling. We mention only certain type of attachments where a coherent labelling exists for the newly constructed 
polyhedron. The sections are divided based on the type of attachment of the tetrahedron.
\subsection{Tetrahedron Attachment with One Face Identified}
\label{sec:T1F}
~\\
Here we attach a tetrahedron with one face identified to a given polyhedron. Even here there are different types of attachments. We mention them below.
\subsubsection{\bf{Tetrahedron Attachment with One Face Identified and No Edge Vanishes}}
\label{sec:T1FNEV}
~\\
\begin{figure}[h]
	\centering
	\includegraphics[width = 0.4\textwidth]{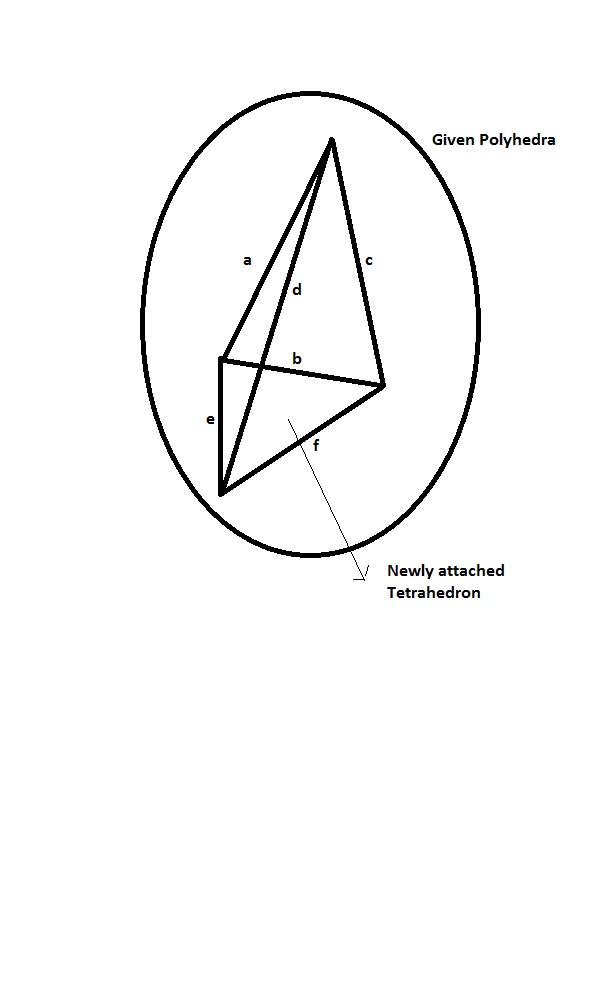}
	\vspace{-10em}
	\caption{Tetrahedron with labels attached to a given Polyhedron
		where one face is identified and no edge vanishes}
	\label{fig:Three}
\end{figure}
Consider a polyhedron with a triangular face with labels $a<b<c$ in the anticlockwise order 
and we are attaching to this face a tetrahedron with labels $abcdef$ as shown in Figure~\ref{fig:Three}
with three new faces. In this construction when the face $a<b<c$ is identified, none of the edges 
vanish. To get a required labelling it is enough that we have 
\equ{f<a<e<b<c<d}
The new faces created are $f<e<b,a<e<d,f<c<d$. This gives the construction shown in Figure~\ref{fig:Three}. 
The actual values for $f,e,d$ can be chosen coherently.

\subsubsection{\bf{Tetrahedron Attachment with One Face Identified and One Edge Vanishes}}
~\\
Consider a polyhedron with a triangular face labelled $abc$ in the anticlockwise order 
and we are attaching to this face a tetrahedron with labels $abcdef$ as shown in Figure~\ref{fig:Four}
with three new faces. Assume that the edge $b$ vanishes. Since an edge of the base plane vanishes we have 
the orientation of the base plane as $z<b<y$ and we assume by a cyclic change of numbers if necessary that 
$z<b<y$ occurs actually.
\begin{figure}[h]
\centering
\includegraphics[width = 0.5\textwidth]{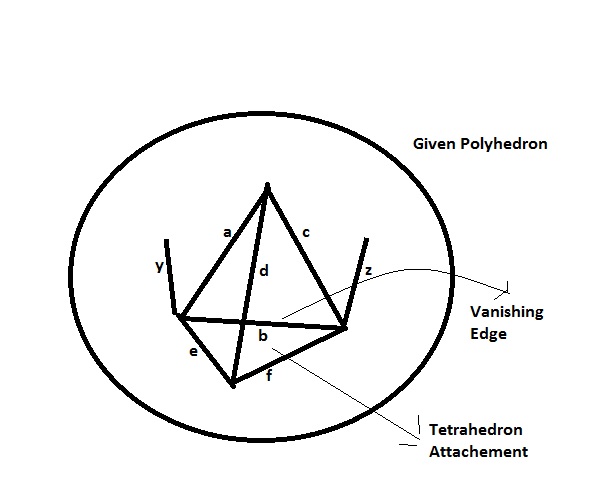}
\caption{Tetrahedron with labels attached to a given Polyhedron
with one edge $b$ vanishes and the edges $a,c$ remain}
\label{fig:Four}
\end{figure}

First assume that $a<b<c$. Now to get a required labelling it is enough that we have 
\equa{&a<f<b<e<c<d\\
&z<f<b<e<y}
The new faces created are $a<e<d,f<c<d,z<f<e<y$. This gives the construction shown in Figure~\ref{fig:Four}. 
The actual values for $f,e,d$ can be chosen coherently.

If $b<c<a$ then to get a required labelling it is enough that we have 
\equa{&f<b<e<c<d<a\\
&z<f<b<e<y}
The new faces created are $e<d<a,f<c<d,z<f<e<y$. This gives the construction shown in Figure~\ref{fig:Four}. 
The actual values for $f,e,d$ can be chosen coherently.

If $c<a<b$ then to get a required labelling it is enough that we have 
\equa{&c<d<a<f<b<e\\
&z<f<b<e<y}
The new faces created are $d<a<e,c<d<f,z<f<e<y$. This gives the construction shown in Figure~\ref{fig:Four}. 
The actual values for $f,e,d$ can be chosen coherently.

\subsubsection{\bf{Tetrahedron Attachment with One Face Identified and Two Edges and their Common Vertex also Vanish}}
~\\
\begin{figure}[h]
\centering
\includegraphics[width = 0.6\textwidth]{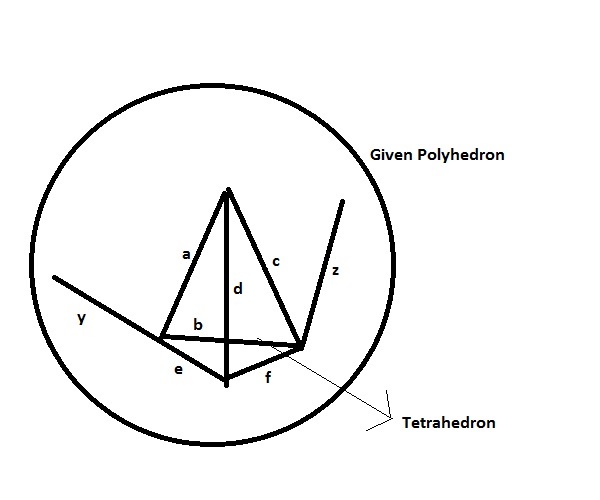}
\caption{Tetrahedron with labels attached to a given Polyhedron
with two edges $a,b$ vanish with their common vertex $a\cap b$ also vanishes and the edge $c$ remains}
\label{fig:Five}
\end{figure}
Consider a polyhedron with a triangular face labelled $abc$ in the anticlockwise order 
and we are attaching to this face a tetrahedron with labels $abcdef$ as shown in Figure~\ref{fig:Five}
with three new faces. Assume that the edges $a,b$ vanish. 

Suppose the edges $e,y$ are the same edges i.e. $e$ is obtained by extending $y$.
The vertex which is the meeting point of $e,y,a,b$ vanishes after the tetrahedron attachment.
Here we just assign new values to $d,e,f$ as 
\begin{itemize}
\item $e=y$ infact the same edge.
\item $d=a$ because $a$ vanishes. 
\item $f=b$ because $b$ vanishes.
\end{itemize}
The new faces created in the anticlockwise order are given by 
\begin{itemize}
\item $\cdots \lra e=y\lra d=a\lra \cdots $,
\item $\cdots \lra z\lra f=b\lra e=y\lra \cdots$,
\item $d=a\lra f=b\lra c\lra d=a$.
\end{itemize}
These are all coherently labelled.
\subsubsection{\bf{Tetrahedron Attachment with Three Edges Vanishing (Capping off a triangle with a Tetrahedron) with Three Vertices also Vanish}}
\label{sec:TATEVTVV}
~\\
Here there are many possibilities while attaching a tetrahedron depending on the vanishing of the vertices. 
Here we consider only one scenario i.e. when the vertices of the attaching triangular face 
also vanish along with the three edges. In this case it is even easier as we remove these 
three edges and their labels and for the extended edges we retain the same labels respectively.

\subsection{Tetrahedron Attachment with Two Faces Identified}
\label{sec:T2F}
We just consider one scenario here in this section.
\subsubsection{\bf{Tetrahedron Attachment with Two Faces Identified and their Common Edge Vanishes
with its Opposite Edge gets a Label and the remaining Four Edges do not Vanish}}
\label{sec:T2FI}
~\\
\begin{figure}[h]
\centering
\includegraphics[width = 0.5\textwidth]{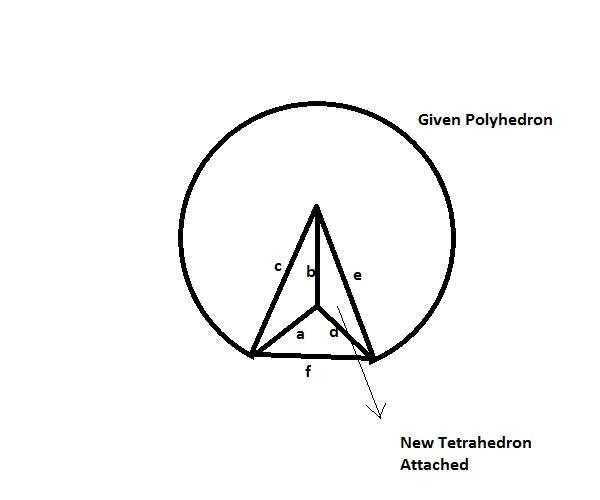}
\caption{Tetrahedron with labels attached to a given Polyhedron where 
two faces $abc,deb$ are attached with their common edge $b$ vanishes, two new faces $adf,fec$ are formed
with the edge $f$ gets a new label}
\label{fig:Six}
\end{figure}
Referring to Figure~\ref{fig:Six} in here we attach a tetrahedron with two of its faces identified
to a given polyhedron.
Consider a tetrahedron with edges $abcdef$ such that the face $abc$ and the face $deb$ are attached to the 
given polyhedron. One edge $b$ vanishes. Moreover two new faces $adf,fec$ are formed.
Here we need a coherent value for $f$ so that the newly formed faces $adf$, $fec$ gets a choice
of increasing labels with anticlockwise orientation. 

One among each of the following sets of inequalities occur.
\equa{&a<b<c,b<c<a,c<a<b\\
&d<e<b,e<b<d,b<d<e}

Now we have four cases \equ{c<e,a<d\text{ or }c<e,d<a\text{ or }e<c,a<d\text{ or }e<c,d<a.}
We prove that for each of the new faces formed we can choose a coherent value for $f$.

We need one among each of the following sets of inequalities to occur.
\equa{&a<d<f\text{ or }d<f<a\text{ or }f<a<d\\
&f<e<c\text{ or }e<c<f\text{ or }c<f<e}

Suppose $c<e,a<d$ then we have one of the following with a coherent choice for $f$ also.
\begin{enumerate}
\item $c<e<a<d$ with $c<f<e<a<d$.
\item $c<a<e<d$ with $c<f<a<e<d$.
\item $c<a<d<e$ with $c<f<a<d<e$.
\item $a<c<d<e$ with $a<c<d<f<e$.
\item $a<d<c<e$ with $a<d<c<f<e$.
\item $a<c<e<d$. This case does not occur as there is no value of coherent $b$ which satisfies. We do not
need to consider this case at all.
\end{enumerate}
Suppose $c<e,d<a$ then we have one of the following with a coherent choice for $f$ also.
\begin{enumerate}
\item $c<e<d<a$. This case does not occur as there is no value of coherent $b$ which satisfies. We do not
need to consider this case at all.
\item $c<d<e<a$ with $c<d<f<e<a$.
\item $c<d<a<e$ with $c<d<f<a<e$.
\item $d<c<a<e$ with $d<c<f<a<e$.
\item $d<a<c<e$. This case does not occur as there is no value of coherent $b$ which satisfies. We do not
need to consider this case at all.
\item $d<c<e<a$ with $d<c<f<e<a$.
\end{enumerate}
The remaining two cases $e<c,a<d$ or $e<c,d<a$ are also similar.
Suppose we have $e<c,a<d$ then we choose $f$ to be a largest element so that $e<c<f,a<d<f$.
Suppose we have $e<c,d<a$ then we have one of the following with a coherent choice for $f$ also.
\begin{enumerate}
\item $e<c<d<a$ with $e<c<d<f<a$.
\item $e<d<c<a$ with $e<d<c<f<a$.
\item $d<e<c<a$ with $d<e<c<f<a$.
\item $d<e<a<c$ with $d<f<e<a<c$.
\item $e<d<a<c$. This case does not occur as there is no value of coherent $b$ which satisfies. We do not
need to consider this case at all.
\item $d<a<e<c$ with $d<f<a<e<c$.
\end{enumerate}
This completes all the possibilities in this case.

\subsection{\bf{Construction of Certain Class of Polyhedra with Tetrahedra}}
We prove a theorem on existence of coherent labelling for certain class of polyhedra.
\begin{theorem}
\label{theorem:PolyhedronConstruction}
Let $\mcl{P}$ be a polyhedron for which a coherent labelling already exists. 
Suppose a new polyhedron is constructed using tetrahedron attachment using any of the following type of attachments.
\begin{itemize}
\item A tetrahedron is attached with one face identified and no edge vanishes.
\item A tetrahedron is attached with one face identified and only one edge vanishes.
\item A tetrahedron is attached with one face identified and two edges and their common vertex vanish.
\item A tetrahedron is attached with one face identified and three edges and the three vertices also vanish.
\item A tetrahedron is attached with two faces identified and their common edge vanishes and its opposite
edge gets a label with the remaining four edges do not vanish.
\end{itemize}
Then the newly constructed polyhedron also has a coherent labelling.
\end{theorem}
\begin{proof}
The proof of this theorem follows from various types of attachments given in~\ref{sec:T1F},~\ref{sec:T2F}.
\end{proof}
\section{\bf{Coherent Labelling of Bipyramids}}
This section considers coherent labelling of bipyramids as an application of
Theorem~\ref{theorem:PolyhedronConstruction}.
Here we prove a theorem below. 
\begin{theorem}
\label{theorem:Bipyramids}
Let $\mcl{P}$ be a polyhedron which is a bipyramid over an $n$-gon. Then the edges of this polyhedron 
can be labelled coherently such that we have a cyclic choice of increasing anticlockwise labellings 
corresponding to an outward pointing normal for every face.
\end{theorem}
\begin{proof}
First we construct a pyramid with a coherent labelling as given in Section~\ref{sec:CLIP} and then we 
start attaching tetrahedra to the base of the pyramid for the construction of bipyramid using non-intersecting diagonals of the base
$n$-gon. We start with the previous construction given in Section~\ref{sec:T1FNEV} and then perform a sequence of constructions as given in Section~\ref{sec:T2FI}.
This completes the construction of the bipyramid with a coherent labelling.
\end{proof}
\begin{example}[Coherent Labelling of an Octahedron]
~\\
We can construct a coherent labelling for an octahedron as a bipyramid using the above tetrahedra attachments as in the previous theorem. A coherent labelling of an octahedron is given in Figure~\ref{fig:Seven}.
\begin{figure}[h]
	\centering
	\includegraphics[width = 0.3\textwidth]{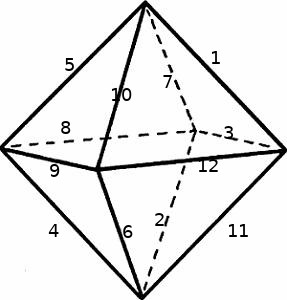}
	\caption{A Labelling of an Octahedron}
	\label{fig:Seven}
\end{figure}

We have thus far proved, for pyramids (which include tetrahedron), 
bipyramids (which include octahedron), that there exists a coherent labelling. 
\end{example}
\section{\bf{Labelling of Kleetopes}}
\label{sec:LPP}
~\\
Here in this section we describe one more class of polyhedra which can be coherently labelled.
Actually here we prove something more. We begin with a definition.
\begin{defn}
\label{defn:PP}
Let $\mcl{P}$ be any polyhedron in three dimensional Euclidean Space. For any face $F$ of $\mcl{P}$
with $n$-edges we attach a pyramid with $n$-gon as a base to extend the polyhedron by keeping all the
edges of the faces $F$ intact. A Kleetope over $\mcl{P}$ is a polyhedron $\mcl{Q}$ obtained
by attaching suitable pyramids to all the faces of the polyhedron $\mcl{P}$ with this type of
attachment for each face.
\end{defn}
\begin{example}
An example of a Kleetope polyhedron is as shown in Figure~\ref{fig:Eight}. This is obtained from a dodecahedron and it is called a stellated dodecahedron or a pentakis dodecahedron. For more about polyhedral models refer to M.~J.~Wenninger~\cite{MR0467493},~\cite{MR0730208}.
\begin{figure}[h]
	\centering
	\includegraphics[width = 0.3\textwidth]{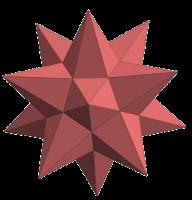}
	\caption{A Stellated Dodecahedron, A Kleetope}
	\label{fig:Eight}
\end{figure}
\end{example}
Now we prove the following theorem.
\begin{theorem}
\label{theorem:LPP}
~\\
Let $\mcl{P}$ be any polyhedron with an arbitrary labelling of all the edges. Then there exists a coherent labelling of the Kleetope $\mcl{Q}$ of $\mcl{P}$ extending the given labelling of the edges of $\mcl{P}$.
\end{theorem}
\begin{proof}
We prove this theorem by extending labelling to the pyramid $\mcl{P}_F$ in $\mcl{Q}$ for each face 
$F$ of the polyhedron $\mcl{P}$. Consider the arbitrary labelling of the face $F$ given by integers
\equ{a_1\lra a_2\lra \cdots \lra a_n \lra a_1} with $a_1$ as the least label in this anticlockwise order.
Let $P_F$ be the apex of the pyramid for this face in the polyhedron $\mcl{Q}$. Let 
\equ{b_{12},b_{23},\cdots,b_{n1}} are the labels of the edges joining a vertex of the face $F$ and the apex 
$P_F$. Now after attachment the face $F$ vanishes and the new triangular faces with labels are given by
\equ{\{b_{n1},a_1,b_{12}\},\{b_{12},a_2,b_{23}\},\cdots,\{b_{(n-2)(n-1)},a_{n-1},b_{(n-1)n}\},\{b_{(n-1)n},a_n,b_{n1}\}.}
We need to prove that there exists a coherent choice for the labels $b_{ij}$ which initially may be fractions
and need not be integers. 

This is an inductive procedure as follows. In the first step we have $a_1<a_2$ as $a_1$ is the least
label. Hence choose $b_{12}$ such that $a_1<b_{12}<a_1+1\leq a_2$. Now in the second step there are two
possibilities. If $a_2<a_3$ then choose $b_{23}$ such that $a_2<b_{23}<a_2+1\leq a_3$. If $a_2>a_3$ then 
choose $b_{23}$ such that $a_1<b_{23}<b_{12}<a_1+1\leq a_3<a_2$. Inductively we proceed with the labelling. 
Suppose $b_{(i-1)i}$ has a coherent value then we need to choose a value for $b_{i(i+1)}$. Here we have two 
choices. Suppose $a_{i}<a_{i+1}$. Then we choose $b_{i(i+1)}$ such that $a_i<b_{i(i+1)}<a_i+1\leq a_{i+1}$.
Hence we have 
\equ{b_{(i-1)i}<a_i<b_{i(i+1)}.}
Suppose $a_i>a_{i+1}$. Then choose $b_{i(i+1)}$ such that \equ{a_1<b_{i(i+1)}<\min\{b_{12},b_{23},\cdots,b_{(i-1)i}\}<a_1+1.}
Then we have by choice $b_{i(i+1)}<b_{(i-1)i}$ and by induction $b_{(i-1)i}<a_i$ and hence
\equ{b_{i(i+1)}<b_{(i-1)i}<a_i.}
Again here we observe that $b_{i(i+1)}<a_{i+1}$ and the induction step is completed.
Now we arrive at the last step of choosing a coherent label for $b_{n1}$. 
Choose $b_{n1}$ such that $b_{n1}>\max\{a_n,b_{12}\}$ then we have 
\equ{b_{(n-1)n}<a_n<b_{n1},a_1<b_{12}<b_{n1}.}
This completes the proof of this theorem.
\end{proof}
In the method of proof of Theorem~\ref{theorem:LPP} we have proved the following extension theorem as well.
\begin{theorem}
\label{theorem:OnePyramid}
Let $\mcl{P}$ be any polyhedron with a coherent labelling of all the edges giving rise to anticlockwise oriented faces of $\mcl{P}$ except possibly for a face $F$ of $\mcl{P}$ with $n_{F}$ edges where an anticlockwise increasing pattern of inequalities of labels need not exist. Let $\mcl{Q}$ be the polyhedron obtained by adding a pyramid over an $n_{F}$-gon whose base is attached to $F$ such that the face $F$ vanishes in $\mcl{Q}$ but none of the edges of $F$ vanish in $\mcl{Q}$. Then there exists a coherent labelling of the polyhedron $\mcl{Q}$ extending the given labelling of the edges of $\mcl{P}$.
\end{theorem}
\begin{note}
	Using the proof of the Theorem~\ref{theorem:LPP} we can obtain that any 
	pyramid and bipyramid can be labelled coherently.
\end{note}

\section{\bf{Class of Polyhedra which are not Coherently Labellable}}
We begin with a couple of counter examples and then prove Theorem~\ref{theorem:ChopPyramid} where a certain class of polyhedra are not coherently labellable. 
\subsection{Counterexample: Impossible to Label a Cube or Cuboid Coherently}
~\\
This section considers cubes or cuboids.
We can try using Theorem~\ref{theorem:PolyhedronConstruction} as follows.
We have the cube or cuboid $[v_1v_2v_3v_4v_5v_6v_7v_8]$ as shown in Figure~\ref{fig:Nine}.

\begin{enumerate}
	\item First we start with  tetrahedron $[v_1v_6v_7v_8]$.
	\item Then attach tetrahedron $[v_1v_5v_6v_8]$. This is a one face $[v_1v_6v_8]$ attachment with one edge 
	$[v_6v_8]$ vanishes.
	\item Then attach tetrahedron $[v_1v_4v_5v_6]$. This is a one face $[v_1v_5v_6]$ attachment with one edge 
	$[v_1v_6]$ vanishes.
	\item Then attach tetrahedron $[v_1v_3v_4v_5]$. This is a one face $[v_1v_4v_5]$ attachment with one edge 
	$[v_4v_5]$ vanishes.
	\item Then attach tetrahedron $[v_1v_3v_5v_8]$. This is a two face attachment on faces $[v_1v_3v_5]$,
	$[v_1v_5v_8]$ and 
	edge $[v_1v_5]$ vanishes whose opposite edge is $[v_3v_8]$ which gets a label. Moreover the five edges
	$[v_1v_8],[v_1v_3],[v_3v_8],[v_3v_5],[v_5v_8]$ remain intact with their labels. This attachment is an 
	application of the attachment described in Section~\ref{sec:T2FI}.
\end{enumerate}
\begin{figure}[h]
	\centering
	\includegraphics[width = 0.4\textwidth]{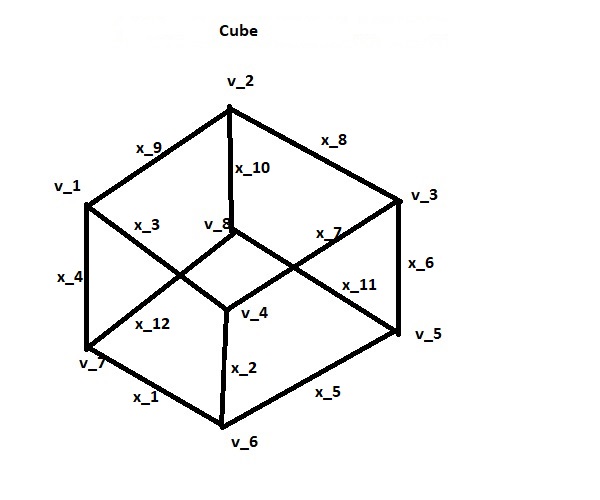}
	\caption{A Cube or a Cuboid}
	\label{fig:Nine}
\end{figure} 
So far we have obtained a coherent labelling.
Now we have one more attachment which is to attach $[v_1v_2v_3v_8]$. 
This is a one face $[v_1v_3v_8]$ attachment with three 
edges $[v_1v_3],[v_1v_8],[v_3v_8]$ vanish. The vertices $v_1,v_3,v_8$ remain and do not vanish. Hence Theorem
~\ref{theorem:PolyhedronConstruction} is not applicable.
We observe as a consequence of the theorem below it is not possible to construct a cube or a cuboid with  
those attachments of tetrahedra where the coherent labelling can be extended.
Now we prove the following theorem.
\begin{theorem}
	\label{theorem:Cube}
	There does not exist a coherent labelling for a cube or a cuboid.
\end{theorem}
\begin{proof}
	By referring to Figure~\ref{fig:Nine} let us label the edges from $x_1,\cdots,x_{12}$.
	The faces with possible inequalities are given by 
	\equa{F_1: x_1<x_2<x_3<x_4\mid x_2<x_3<x_4<x_1&\mid x_3<x_4<x_1<x_2\mid\\& x_4<x_1<x_2<x_3\\
		F_2: x_5<x_6<x_7<x_2\mid x_6<x_7<x_2<x_5&\mid x_7<x_2<x_5<x_6\mid\\& x_2<x_5<x_6<x_7\\
		F_3:x_{11}<x_{10}<x_8<x_6\mid x_{10}<x_8<x_6<x_{11}&\mid x_8<x_6<x_{11}<x_{10}\mid\\& x_6<x_{11}<x_{10}<x_8\\
		F_4:x_{12}<x_4<x_9<x_{10}\mid x_4<x_9<x_{10}<x_{12}&\mid x_9<x_{10}<x_{12}<x_4\mid\\& x_{10}<x_{12}<x_4<x_9\\
		F_5:x_3<x_7<x_8<x_9\mid x_7<x_8<x_9<x_3&\mid x_8<x_9<x_3<x_7 \mid\\& x_9<x_3<x_7<x_8\\
		F_6:x_1<x_{12}<x_{11}<x_5\mid x_{12}<x_{11}<x_5<x_1&\mid x_{11}<x_5<x_1<x_{12}\mid\\ & x_5<x_1<x_{12}<x_{11}}
	
	Now we prove that there is no choice of inequalities for the labels $x_1,\cdots,x_{12}$
	which extends six strings of inequalities with one possibility for each of the six faces.
	
	Assume without loss of generality that we choose the first possibility for the face $F_1$
	i.e. $F_1: x_1<x_2<x_3<x_4$.
	Now we dispose off each possibility as impossible for the face $F_2$.
	\begin{enumerate}
		\item Consider $F_2:x_5<x_6<x_7<x_2$. Then we have $x_7<x_3$. So only $F_5: x_7<x_8<x_9<x_3$ occurs
		and the remaining possibilities for the face $F_5$ do not occur. Now we have $x_6<x_8$. So only 
		$F_3:x_6<x_{11}<x_{10}<x_8$ occurs and the remaining possibilities for the face $F_3$ do not occur.
		Now $x_5<x_{11}$. So only $F_6:x_5<x_1<x_{12}<x_{11}$ occurs and the remaining possibilities do not occur.
		Now we conclude that $x_{10}<x_9<x_4$. This rules out all the possibilities for the face $F_4$.
		
		\item Consider $F_2:x_6<x_7<x_2<x_5$. Then we have $x_7<x_3$. So only $F_5: x_7<x_8<x_9<x_3$ occurs
		and the remaining possibilities for the face $F_5$ do not occur. Now we have $x_6<x_8$. So only 
		$F_3:x_6<x_{11}<x_{10}<x_8$ occurs and the remaining possibilities for the face $F_3$ do not occur.
		Now $x_1<x_5$. So only $F_6:x_1<x_{12}<x_{11}<x_5$ occurs and the remaining possibilities for the face $F_6$ 
		do not occur. Similarly $x_9<x_4$. So only $F_4:x_9<x_{10}<x_{12}<x_4$ occurs and the remaining possibilities
		for the face $F_4$ do not occur. Now we conclude $x_{12}<x_{11}<x_{10}<x_{12}$ which is a contradiction.
		
		\item Consider $F_2:x_7<x_2<x_5<x_6$. Then we have $x_7<x_3$. So only $F_5: x_7<x_8<x_9<x_3$ occurs
		and the remaining possibilities for the face $F_5$ do not occur. Now we have $x_9<x_4$. So only 
		$F_4:x_9<x_{10}<x_{12}<x_4$ occurs and the remaining possibilities for the face $F_4$ do not occur.
		Now we have $x_1<x_5$. So only $F_6:x_1<x_{12}<x_{11}<x_5$ occurs and the remaining possibilities for the face $F_6$ 
		do not occur. Similarly we have $x_{11}<x_6$. So only $F_3:x_{11}<x_{10}<x_8<x_6$ occurs and the remaining possibilities 
		for the face $F_3$ do not occur. Now we conclude $x_{12}<x_{11}<x_{10}<x_{12}$ which is a contradiction.
		
		\item Consider $F_2:x_2<x_5<x_6<x_7$. Then we have $x_1<x_5$. So only $F_6:x_1<x_{12}<x_{11}<x_5$ occurs and the 
		remaining possibilities for the face $F_6$ do not occur. Now we have $x_{11}<x_6$. So only $F_3:x_{11}<x_{10}<x_8<x_6$ 
		occurs and the remaining possibilities for the face $F_3$ do not occur. We have $x_8<x_7$. So only $F_5:
		x_8<x_9<x_3<x_7$ occurs and the remaining possibilities for the face $F_5$ do not occur. Now we conclude that 
		$x_{12}<x_{10}<x_9$. This rules out all the possibilities for the face $F_4$.
	\end{enumerate}
	This proves the theorem.
\end{proof}
\subsection{Counterexample: Impossible to Label a Triangular Prism Coherently}
~\\
Now we prove that a triangular prism is not coherently labellable. 
We first mention a definition.
\begin{defn}
\label{defn:CLPVertex}
~\\
We say a labelling of the edges of a polyhedron $\mcl{P}$ is anticlockwise coherent or clockwise coherent at a vertex $v$ of $\mcl{P}$, if the labelling of the edges of $\mcl{P}$ with integers, is such that, at the vertex $v$, we have an anticlockwise traversal or a clockwise traversal respectively, of edges emanating at the vertex $v$, with an increasing choice of labels, except the first edge and the last edge. 
\end{defn}
Now we prove a lemma.
\begin{lem}
\label{lem:TVO}
For any coherent labelling of the tetrahedron, the labelling of the edges of the tetrahedron is clockwise coherent at any of its vertices.
\end{lem}
\begin{proof}
This follows from Theorem~\ref{theorem:Tetrahedron}.	
\end{proof}
\begin{remark}
It is not true that for any coherent labelling of a octahedron, the labelling of the edges of the octahedron is clockwise coherent at a vertex. In fact, at a vertex, the labelling need not be either clockwise coherent or anticlockwise coherent. For example there is no coherency at the topmost and leftmost vertices in the coherently labelled octahedron in Figure~\ref{fig:Seven}.
\end{remark}
Now with the remark the theorem below follows.
\begin{theorem}
A coherent labelling of a polyhedron need not induce either clockwise coherency or an anticlockwise at a vertex, that is, the labellings of the edges incident at a vertex can be incoherent.
\end{theorem}
Now we make another important observation about the tetrahedron in the following lemma.
\begin{lem}
	\label{lem:TetraChopPrism}
For any coherent labelling of the tetrahedron, if we construct a triangular prism by chopping a vertex of a tetrahedron with newly labelled edges $a<b<c$ then the orientation that this labelling gives rise to, for the new face, is always clockwise.
\end{lem}
\begin{proof}
First we observe that if we chop a vertex of a tetrahedron then we obtain a triangular prism. If we chop any vertex of a tetrahedron and construct a triangular prism with newly labelled edges $a<b<c$ then it is a straight forward observation using Lemma~\ref{lem:TVO}, that, this gives rise to a clockwise orientation of the new face in the triangular prism.
\end{proof}
Now we prove that a triangular prism cannot be labelled coherently. 
\begin{lem}
A triangular prism cannot be labelled coherently.
\end{lem}
\begin{proof}
Suppose there is a coherent labelling of a triangular prism. Then by collapsing a triangular face to a point or by attaching a tetrahedron to a triangular face as in Section~\ref{sec:TATEVTVV} with three edges and three vertices vanish, we obtain a coherent labelling of a tetrahedron. Hence by reversing the process, that is, by chopping a suitable vertex we obtain back the triangular prism and has a coherent labelling. This contradicts Lemma~\ref{lem:TetraChopPrism}. Hence a triangular prism cannot be coherently labelled.
\end{proof}
\subsection{\bf{Obstruction Criterion}}
We have now observed the proof of the fact that a triangular prism cannot be coherently labelled. This proof reveals a principle which is stated as a theorem below.
\begin{theorem}
\label{theorem:OC}
Let $\mcl{P}$ be a convex polyhedron. Let $\mcl{Q}$ be a convex polyhedron which is obtained by attaching the base of a pyramid having apex $v$ to a face $F$ of $\mcl{P}$ so that all the edges and vertices of the face $F$ vanish. We assume the following.
\begin{itemize}
\item $\mcl{Q}$ is coherently labellable.
\item In all the coherent labellings of $\mcl{Q}$, the labellings are clockwise coherent at the vertex $v$ of $\mcl{Q}$.
\end{itemize}
Then the polyhedron $\mcl{P}$ which is obtained by chopping vertex $v$ of $\mcl{Q}$ is not coherently labellable.
\end{theorem}
\begin{proof}
Suppose the polyhedron $\mcl{P}$ has a coherent labelling. Then by a pyramid attachment to the face $F$ of $\mcl{P}$ we obtain a coherent labelling on $\mcl{Q}$. Hence by the hypothesis, the labelling of edges is clockwise coherent at the vertex $v$. Now if we chop the vertex $v$ then we can revert back to the initial labels of $\mcl{P}$ which is coherent to begin with. Now we arrive at a contradiction by proving the following claim.

\begin{claim}
For any labelling of $\mcl{Q}$, if we chop the vertex $v$ of $\mcl{Q}$ and label the edges of the face $F$ of $\mcl{P}$ such that we have a coherent labelling on all the faces of $\mcl{P}$ except possibly $F$, then the increasing order of any such labelling of edges of $F$ induces a clockwise orientation for the face $F$.
\end{claim}

Hence, using the claim, we immediately arrive at the conclusion that the polyhedron $\mcl{P}$ is not coherently labellable. 

\begin{proof}[Proof of Claim]
Let $a_1<a_2<\cdots<a_n$ be the labels of the edges incident to the vertex $v$ of $\mcl{Q}$ in the clockwise order. Let $b_1,b_2,\cdots,b_n$ be labels of the edges of the face $F$ such that we have the following $n$-triangular faces for the pyramid attached to $\mcl{P}$
\equ{\{a_1,a_2,b_1\},\{a_2,a_3,b_2\},\cdots,\{a_{n-1},a_n,b_{n-1}\},\{a_n,a_1,b_n\}.}
Assume without loss of generality that $a_1$ is the least label among all the edges of $\mcl{P}$ because the coherency property does not change by cyclically relabelling all the edges of $\mcl{P}$.

Since the we have a coherent labelling (anticlockwise) on all the faces of $\mcl{P}$ except possibly the face $F$. We conclude the following.

\begin{itemize}
\item $a_1<a_n<b_n,a_1<b_1<a_2$ since $a_1$ is the least label. 
\item Either $a_2<b_2<a_3$ or $a_3<a_2<b_2$. But $a_2<a_3$ hence we must have $a_2<b_2<a_3$.
\item Similarly we have $a_i<b_i<a_{i+1}$ for $1\leq i\leq n-1$.
\item Hence we conclude that $b_1<b_2<\cdots<b_{n-1}<b_n$. This induces a clockwise orientation on the face $F$ of the polyhedron $\mcl{P}$.  	
\end{itemize}
This proves the claim.
\end{proof}
Hence we have completed the proof of Theorem~\ref{theorem:OC}.
\end{proof}
As an application of Theorem~\ref{theorem:OC} we prove the following theorem.
\begin{theorem}
\label{theorem:ChopPyramid}
Any pyramid with its apex chopped off cannot be coherently labellable. In particular the triangular prism and the cube or the cuboid cannot be coherently labellable.
\end{theorem}
\begin{proof}
It is enough to prove that any coherent labelling of the pyramid induces a labelling of the edges which is clockwise coherent at the apex. This is definitely true for a tetrahedron using Theorem~\ref{theorem:Tetrahedron}. Now we prove by induction on the number of edges of the base of the pyramid. 

Let $a_1,a_2,\cdots,a_n$ be the labels of the edges incident at the apex in the anticlockwise order with $a_1$ being the least label among all the labels of the edges of the pyramid. Let $b_1,b_2,\cdots,b_n$ be the labels of the edges of the base such the triangular faces are given by 
\equ{\{a_1,a_2,b_1\},\{a_2,a_3,b_2\},\cdots,\{a_{n-1},a_n,b_{n-1}\},\{a_n,a_1,b_n\}.}
Then we first have $a_1<a_n<b_n,a_1<b_1<a_2$ since $a_1$ is the least label.
Now we have one of the following three cases.
\equ{I: a_{n-1}<b_{n-1}<a_n;\ \ II: a_n<a_{n-1}<b_{n-1};\ \ III: b_{n-1}<a_n<a_{n-1}.}
In the cases $I,II$ we have $a_1<a_{n-1}<b_{n-1}$. So we delete $a_n$ and $b_n$ in these two cases to obtain a coherent labelling on a pyramid with a base an $(n-1)-$gon. Hence by induction we have \equ{a_1<a_{n-1}<a_{n-2}<\cdots<a_3<a_2} a labelling which is clockwise coherent at the apex of the pyramid with a base an $(n-1)-$gon. 

In the cases $I,III$ we have $b_{n-1}<b_n$. Since the labelling on the pyramid is coherent we must have 
\equ{b_{n-1}<b_{n-2}<\cdots<b_2<b_1<b_n.}
Moreover in case $III$ we have $a_1<a_{n-1}<L$ where $L$ is any positive integer larger than all the labels $a_i,b_i,1\leq i\leq n$. Now deleting $a_n,b_n$ and relabelling $b_{n-1}$ by $L$ we obtain a coherent labelling on a pyramid with a base an $(n-1)-$gon with labels $a_i,1\leq i\leq n-1,b_j,1\leq j\leq n-2$ and $L$. Here the faces with label $L$ are given by $a_1<a_{n-1}<L$ and 
\equ{b_{n-2}<b_{n-3}<\cdots<b_2<b_1<L.}
In this case $III$ also, now, by induction we have 
\equ{a_1<a_{n-1}<a_{n-2}<\cdots<a_3<a_2} a labelling which is clockwise coherent at the apex of the pyramid with a base an $(n-1)-$gon. .

Combining with the inequality $a_1<a_n<a_{n-1}$ in cases $II$ and $III$ we have 
\equ{a_1<a_n<a_{n-1}<a_{n-2}<\cdots<a_3<a_2} a labelling which is clockwise coherent at the apex of the pyramid with a base an $n$-gon.

Now we consider case $I$. Here we have $a_1<a_n<b_n;\ \ a_1<b_1<a_2;\ \ a_{n-1}<b_{n-1}<a_n$ and $b_{n-1}<b_{n-2}<\cdots<b_2<b_1<b_n$.

Now either $a_3<a_2<b_2$ or $b_2<a_3<a_2$. Now $a_3<a_2<b_2 \Ra b_1<a_2<b_2$ which is a contradiction. So $b_2<a_3<a_2$. For $2\leq i\leq n-3$ if we have if $b_i<a_{i+1}<a_i$ then 
$a_{i+2}<a_{i+1}<b_{i+1} \Ra b_i<a_{i+1}<b_{i+1}$ which is a contradiction for $1\leq i\leq n-2$. Hence we must have $b_{i+1}<a_{i+2}<a_{i+1}$. So inductively 
we obtain $b_i<a_{i+1}<a_i$ for $2\leq i\leq n-2$. For $i=n-2$ we have $b_{n-2}<a_{n-1}<a_{n-2}$. Now using $a_{n-1}<b_{n-1}<a_n$ we obtain $b_{n-2}<a_{n-1}<b_{n-1}$ which is a contradiction. Hence case $I$ does not arise at all.

So we have proved that if we have a coherent labelling of the pyramid then labels induce clockwise coherency at the apex.

Hence we have completed the proof of Theorem~\ref{theorem:ChopPyramid}.  
\end{proof}
\section{\bf{The Dodecahedron and the Icosahedron}}
In this section we prove that the dodecahedron is not coherently labellable (Lemma~\ref{lemma:Dodecahedron}) and the icosahedron is coherently labellable (Lemma~\ref{lemma:Icosahedron}). As a consequence we have the following theorem.
\begin{theorem}
\label{theorem:PlatonicSolids}
Among the five platonic solids or regular polyhedra, the simplicial polyhedra namely the tetrahedron, the octahedron and the icosahedron are coherently labellable and the other two namely the dodecahedron and the cube are not coherently labellable. 
\end{theorem}
Now we mention the following lemma.
\begin{lem}
\label{lemma:Dodecahedron}
The dodecahedron is not coherently labellable.
\end{lem}
\begin{proof}
Consider the dodecahedron in Figure~\ref{fig:Ten} with labels $a_1,\cdots,a_{30}$ for the edges of the faces $F_1,\cdots,F_{12}$.
\begin{figure}[h]
	\centering
	\includegraphics[width = 1.0\textwidth]{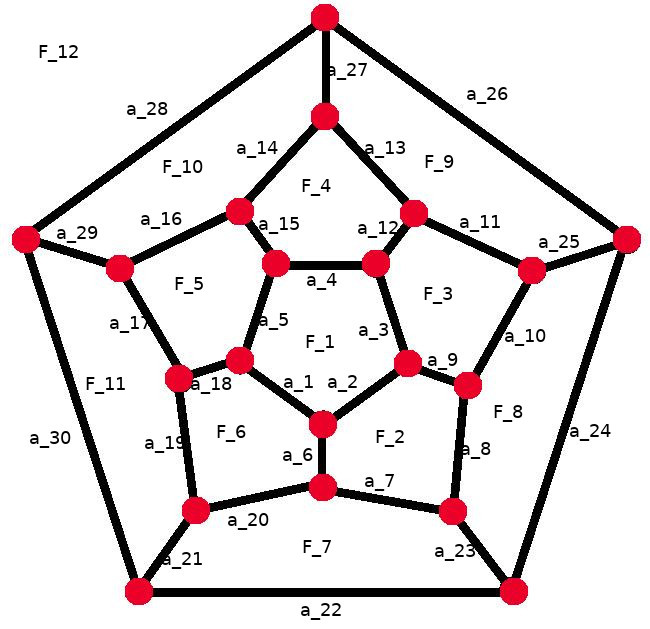}
	\caption{A Dodecahedron}
	\label{fig:Ten}
\end{figure}
Now we prove that the dodecahedron is not coherently labellable in a similar way as given in the proof of Theorem~\ref{theorem:Cube} for the cube.	
The possible choices of strings of inequalities for the faces $F_1,\cdots,F_{12}$ are 
\begin{enumerate}
\item For the face $F_1$
\begin{enumerate}[label=(\Alph*)]	
	\item $a_1<a_2<a_3<a_4<a_5$.
	\item $a_2<a_3<a_4<a_5<a_1$.
	\item $a_3<a_4<a_5<a_1<a_2$.
	\item $a_4<a_5<a_1<a_2<a_3$.
	\item $a_5<a_1<a_2<a_3<a_4$.
\end{enumerate}
\item For the face $F_2$
\begin{enumerate}[label=(\Alph*)]	
	\item $a_2<a_6<a_7<a_8<a_9$.
	\item $a_6<a_7<a_8<a_9<a_2$.
	\item $a_7<a_8<a_9<a_2<a_6$.
	\item $a_8<a_9<a_2<a_6<a_7$.
	\item $a_9<a_2<a_6<a_7<a_8$.				
\end{enumerate}
\item For the face $F_3$
\begin{enumerate}[label=(\Alph*)]	
	\item $a_3<a_9<a_{10}<a_{11}<a_{12}$.
	\item $a_9<a_{10}<a_{11}<a_{12}<a_3$.
	\item $a_{10}<a_{11}<a_{12}<a_3<a_9$.
	\item $a_{11}<a_{12}<a_3<a_9<a_{10}$.
	\item $a_{12}<a_3<a_9<a_{10}<a_{11}$.
\end{enumerate}
\item For the face $F_4$
\begin{enumerate}[label=(\Alph*)]	
	\item $a_4<a_{12}<a_{13}<a_{14}<a_{15}$.
	\item $a_{12}<a_{13}<a_{14}<a_{15}<a_4$.
	\item $a_{13}<a_{14}<a_{15}<a_4<a_{12}$.
	\item $a_{14}<a_{15}<a_4<a_{12}<a_{13}$.			
	\item $a_{15}<a_4<a_{12}<a_{13}<a_{14}$.
\end{enumerate}
\item For the face $F_5$
\begin{enumerate}[label=(\Alph*)]	
	\item $a_5<a_{15}<a_{16}<a_{17}<a_{18}$.
	\item $a_{15}<a_{16}<a_{17}<a_{18}<a_5$.	
	\item $a_{16}<a_{17}<a_{18}<a_5<a_{15}$.
	\item $a_{17}<a_{18}<a_5<a_{15}<a_{16}$.
	\item $a_{18}<a_5<a_{15}<a_{16}<a_{17}$.
\end{enumerate}
\item For the face $F_6$
\begin{enumerate}[label=(\Alph*)]	
	\item $a_1<a_{18}<a_{19}<a_{20}<a_6$.
	\item $a_{18}<a_{19}<a_{20}<a_6<a_1$.	
	\item $a_{19}<a_{20}<a_6<a_1<a_{18}$.
	\item $a_{20}<a_6<a_1<a_{18}<a_{19}$.
	\item $a_6<a_1<a_{18}<a_{19}<a_{20}$.
\end{enumerate}
\item For the face $F_7$
\begin{enumerate}[label=(\Alph*)]	
	\item $a_7<a_{20}<a_{21}<a_{22}<a_{23}$.
	\item $a_{20}<a_{21}<a_{22}<a_{23}<a_7$.
	\item $a_{21}<a_{22}<a_{23}<a_7<a_{20}$.
	\item $a_{22}<a_{23}<a_7<a_{20}<a_{21}$.		
	\item $a_{23}<a_7<a_{20}<a_{21}<a_{22}$.
\end{enumerate}
\item For the face $F_8$
\begin{enumerate}[label=(\Alph*)]	
	\item $a_8<a_{23}<a_{24}<a_{25}<a_{10}$.
	\item $a_{23}<a_{24}<a_{25}<a_{10}<a_8$.
	\item $a_{24}<a_{25}<a_{10}<a_8<a_{23}$.
	\item $a_{25}<a_{10}<a_8<a_{23}<a_{24}$.
	\item $a_{10}<a_8<a_{23}<a_{24}<a_{25}$.				
\end{enumerate}
\item For the face $F_9$
\begin{enumerate}[label=(\Alph*)]	
	\item $a_{11}<a_{25}<a_{26}<a_{27}<a_{13}$.
	\item $a_{25}<a_{26}<a_{27}<a_{13}<a_{11}$.
	\item $a_{26}<a_{27}<a_{13}<a_{11}<a_{25}$.
	\item $a_{27}<a_{13}<a_{11}<a_{25}<a_{26}$.	
	\item $a_{13}<a_{11}<a_{25}<a_{26}<a_{27}$.
\end{enumerate}
\item For the face $F_{10}$
\begin{enumerate}[label=(\Alph*)]	
	\item $a_{14}<a_{27}<a_{28}<a_{29}<a_{16}$.
	\item $a_{27}<a_{28}<a_{29}<a_{16}<a_{14}$.
	\item $a_{28}<a_{29}<a_{16}<a_{14}<a_{27}$.	
	\item $a_{29}<a_{16}<a_{14}<a_{27}<a_{28}$.
	\item $a_{16}<a_{14}<a_{27}<a_{28}<a_{29}$.
\end{enumerate}
\item For the face $F_{11}$
\begin{enumerate}[label=(\Alph*)]	
	\item $a_{17}<a_{29}<a_{30}<a_{21}<a_{19}$.
	\item $a_{29}<a_{30}<a_{21}<a_{19}<a_{17}$.	
	\item $a_{30}<a_{21}<a_{19}<a_{17}<a_{29}$.
	\item $a_{21}<a_{19}<a_{17}<a_{29}<a_{30}$.
	\item $a_{19}<a_{17}<a_{29}<a_{30}<a_{21}$.		
\end{enumerate}
\item For the face $F_{12}$
\begin{enumerate}[label=(\Alph*)]	
	\item $a_{22}<a_{30}<a_{28}<a_{26}<a_{24}$.
	\item $a_{30}<a_{28}<a_{26}<a_{24}<a_{22}$.
	\item $a_{28}<a_{26}<a_{24}<a_{22}<a_{30}$.
	\item $a_{26}<a_{24}<a_{22}<a_{30}<a_{28}$.
	\item $a_{24}<a_{22}<a_{30}<a_{28}<a_{26}$.
\end{enumerate}
\end{enumerate}
Assume without loss of generality that $(1):(A)$ holds. Now we will exhaust all cases $(2):(A)-(E)$. Suppose $(2):(A)$ holds. Then we have 
\equa{&((1):(A)\text{ and } (2):(A)) \Ra (6):(A)\Ra (7):(B)\Ra (8):(B)\Ra (3):(C)\\ &\Ra (9):(B)\Ra (4):(C)\Ra (10):(B)\Ra (5):(C)\Ra (11):(B).}
Now we conclude that $a_{17}\us{(5):(C)}{<}a_{18}\us{(6):(A)}{<}a_{19}\us{(11):(B)}{<}a_{17}$ which is a contradiction.

Suppose $(2):(B)$ holds. Then we have 
\equa{&((1):(A)\text{ and } (2):(B)) \Ra (3):(B)\Ra (4):(B)\Ra (5):(B)\Ra (8):(A)\\ &\Ra
(9):(A)\Ra (10):(A)\Ra (12):(E)\Ra (11):(C)\Ra (7):(D).}
Now we conclude that $a_8\us{(8):(A)}{<}a_{23}\us{(7):(D)}{<}a_7\us{(2):(B)}{<}a_8$ which is a contradiction.

Suppose $(2):(C)$ holds. Then we have 
\equa{&((1):(A)\text{ and } (2):(C)) \Ra (3):(B)\Ra (4):(B)\Ra (5):(B)\Ra (6):(A)\\ &\Ra (8):(A)\Ra (9):(A)\Ra (10):(A)\Ra (11):(A)\Ra (12):(C).}
Now we conclude that $a_{26}\us{(9):(A)}{<}a_{27}\us{(10):(A)}{<}a_{28}\us{(12):(C)}{<}a_{26}$ which is a contradiction.

Suppose $(2):(D)$ holds. Then we have 
\equa{&((1):(A)\text{ and } (2):(D)) \Ra (3):(B)\Ra (4):(B)\Ra (5):(B)\Ra (6):(A)\\ &\Ra
(7):(B)\Ra (8):(A).}
Now we conclude that $a_{10}\us{(3):(B)}{<}a_{12}\us{(4):(B)}{<}a_{15}\us{(5):(B)}{<}a_{18}\us{(6):(A)}{<}a_{20}\us{(7):(B)}{<}a_{23}\us{(8):(A)}{<}a_{10}$ which is a contradiction.

Suppose $(2):(E)$ holds. Then we have 
\equa{&((1):(A)\text{ and } (2):(E)) \Ra (3):(B)\Ra (4):(B)\Ra (5):(B)\Ra (6):(A)\\ &\Ra
(7):(B)\Ra (8):(B).}
Now we conclude that $a_{10}\us{(3):(B)}{<}a_{12}\us{(4):(B)}{<}a_{15}\us{(5):(B)}{<}a_{18}\us{(6):(A)}{<}a_{20}\us{(7):(B)}{<}a_{23}\us{(8):(B)}{<}a_{10}$ which is a contradiction.

This proves that the dodecahedron is not coherently labellable.
\end{proof}
\begin{lem}
\label{lemma:Icosahedron}
The icosahedron is coherently labellable.
\end{lem}
\begin{proof}
It is enough to prove that following polyhedron in Figure~\ref{fig:Eleven} is coherently labellable. If we attach two pentagonal based pyramids to the two opposite pentagonal faces of the polyhedron in Figure~\ref{fig:Eleven}, we get icosahedron. Hence the icosahedron is coherently labellable using Theorem~\ref{theorem:OnePyramid}.
\begin{figure}[h]
	\centering
	\includegraphics[width = 0.3\textwidth]{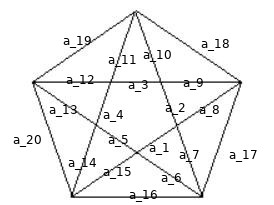}
	\caption{An icosahedron with two opposite vertices removed giving rise to two pentagonal faces for the new polyhedron}
	\label{fig:Eleven}
\end{figure}
Now assume the edges in the Figure~\ref{fig:Eleven} are labelled $a_1,\cdots,a_{20}$. Then we can have the following strings of inequalities satisfied simultaneously.
\begin{itemize}
\item $a_1<a_2<a_3<a_4<a_5$.
\item $a_{16}<a_{20}<a_{19}<a_{18}<a_{17}$.
\item $a_1<a_6<a_7; a_2<a_8<a_9; a_3<a_{10}<a_{11}; a_4<a_{12}<a_{13}; a_5<a_{14}<a_{15}$.
\item $a_{16}<a_6<a_{15}; a_{17}<a_8<a_7; a_{18}<a_{10}<a_9; a_{19}<a_{12}<a_{11}; a_{20}<a_{14}<a_{13}$.	
\end{itemize}
This can be reasoned out as follows.
Let $A=\{a_1,a_2,a_3,a_4,a_5\}, B=\{a_6,a_8,a_{10},a_{12},$ $a_{14}\}, C =\{a_7,a_9,a_{11},a_{13},a_{15}\}, D=\{a_{16},a_{17},a_{18},a_{19},a_{20}\}$.
Choose labels such that $A<B<C$ and $D<B<C$, that is, for any $a\in A,b\in B,c\in C,d\in D$ we have $a<b<c,d<b<c$ in addition to the strings of inequalities $a_1<a_2<a_3<a_4<a_5$ and  $a_{16}<a_{20}<a_{19}<a_{18}<a_{17}$. Hence the lemma follows.	
\end{proof}
Now we mention a definition and a theorem.
\begin{defn}
An n-sided antiprism is a polyhedron composed of two parallel copies of some particular n-sided polygon, connected by an alternating band of triangles.
A gyroelongated bipyramid is a polyhedron obtained by elongating an n-gonal bipyramid by inserting an n-gonal antiprism between its congruent halves. For example the icosahedron is a gyroelongated bipyramid obtained from pentagonal bipyramid.
\end{defn}
\begin{theorem}
Any gyroelongated bipyramid is coherently labellable.
\end{theorem}
\begin{proof}
The proof is similar that of Lemma~\ref{lemma:Icosahedron}.	
\end{proof}
\section{\bf{Simplicial Polyhedra}}
We have observed that apart from the pyramids, the following polyhedra are coherently labellable.
\begin{enumerate}
	\item The Tetrahedron,
	\item Bipyramids which includes the Octahedron,
	\item Kleetopes,
	\item Gyroelongated Bipyramids which includes the Icosahedron.
\end{enumerate}
This leads us to the following theorem which we prove in this section. The converse of this theorem is not true since pyramids on any base polygon are also coherently labellable.
\begin{thmOmega}
\namedlabel{theorem:SP}{$\Gom$}
	Any three dimensional simplicial polyhedron, that is, a polyhedron with all its faces having exactly three sides, is coherently labellable.
\end{thmOmega}
In order to get an idea to prove Theorem~\ref{theorem:SP}, we consider two more examples of simplicial polyhedron namely a snub disphenoid (Figure~\ref{fig:Twelve}) and pentakis snub dodecahedron (Definition~\ref{defn:PSD}).
\begin{figure}[h]
	\centering
	\includegraphics[width = 0.3\textwidth]{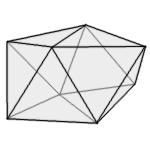}
	\caption{A Snub Disphenoid}
	\label{fig:Twelve}
\end{figure}
Snub Disphenoid is as shown in Figure~\ref{fig:Twelve}. Here there are four vertices of degree five and four vertices of degree eight. If we remove any two non-adjacent vertices of degree five and all the ten edges incident on them, then we have two faces and eight edges. We label these edges so that the two faces get an anticlockwise orientation and then by method of attaching pyramids (here pentagonal pyramids) using Theorem~\ref{theorem:OnePyramid} we get a coherent labelling of snub disphenoid.

\begin{defn}
\label{defn:PSD}
Pentakis snub dodecahedron is a simplicial polyhedron obtained by attaching twelve pentagonal pyramids to a snub dodecahedron which is as given in Figure~\ref{fig:Thirteen}. 	
\begin{figure}[h]
	\centering
	\includegraphics[width = 0.3\textwidth]{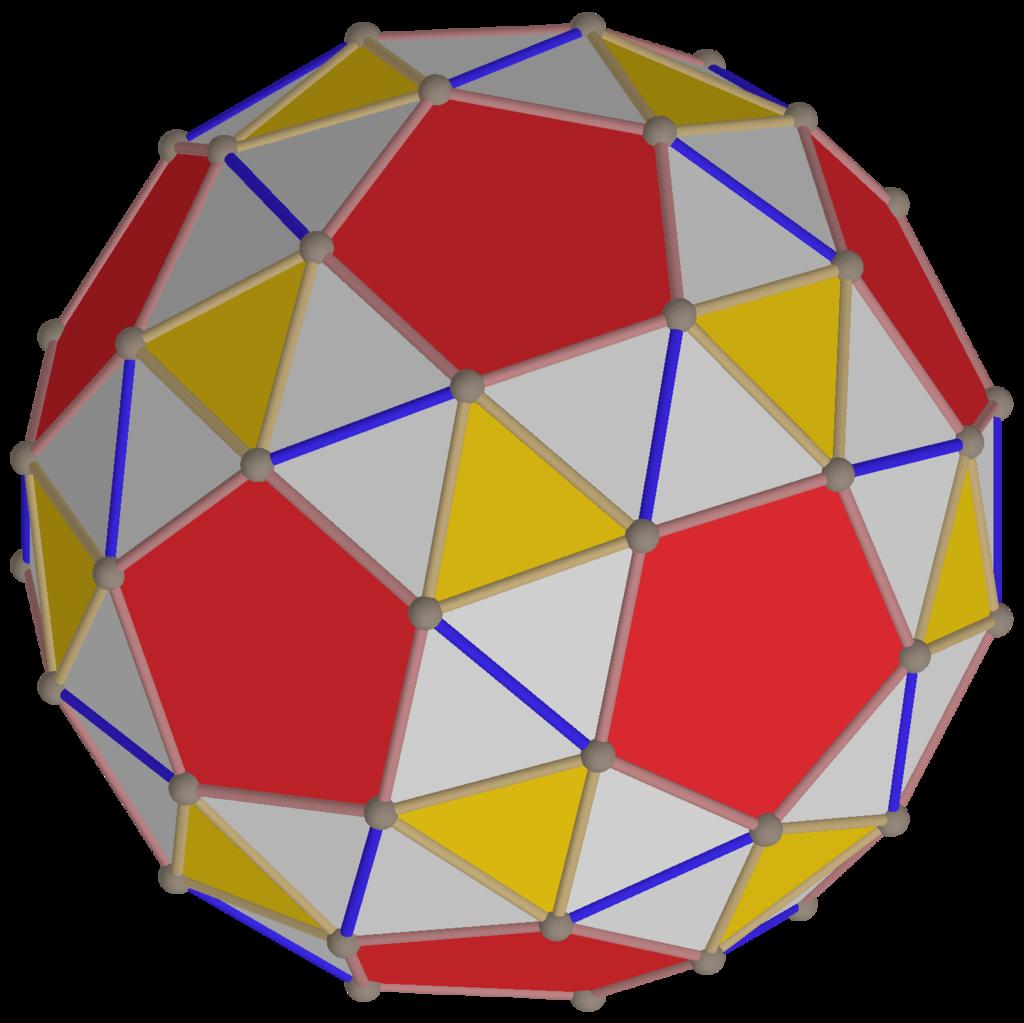}
	\caption{A Snub Dodecahedron}
	\label{fig:Thirteen}
\end{figure}
\end{defn}
In order to label the pentakis snub dodecahedron, we label a snub dodecahedron without the pentagonal faces. For this first we label coherently yellow triangles in Figure~\ref{fig:Thirteen} and then coherently label the white triangles by labelling the blue edges and then labelling edges which are the sides of pentagons. Now using Theorem~\ref{theorem:OnePyramid} and attaching pentagonal pyramids we label coherently the pentakis snub dodecahedron.
\subsection{\bf{Proof of the Main Theorem by the Method of Labelling Partial Pyramids}}
Now we prove Theorem~\ref{theorem:SP} by the method of labelling partial pyramids. 
\begin{proof}
A simplicial polyhedron corresponds to a triangulation of a sphere. Now we remove a maximal set of vertices and the edges incident on them such that the interior of the faces incident on different removed vertices are disjoint. These vertices with edges and faces incident on them correspond just like pyramid attachments as far as labelling is concerned.

Now a sphere with a finite number of open discs removed is made up of certain bands and threads. An example of a thread of edges occurs in snub disphenoid when two non-adjacent vertices of degree five are removed. An example of only bands occuring is seen in a pentakis snub dodecahedron when the pentagonal pyramids are removed. 

After removing a maximal set of vertices, all the remaining vertices are on the boundary of the bands and threads. There are three types of triangles in the remaining figure. 
\begin{enumerate}
	\item Triangles with two sides on the boundary.
	\item Triangles with one side on the boundary.
	\item Triangles with no side on the boundary.
\end{enumerate}
We label the third type of triangles first coherently, that is, those with no side on the boundary. However there is a problem labelling these edges of these triangles coherently as shown below. Consider the following Figure~\ref{fig:Fourteen}.
\begin{figure}[h]
	\centering
	\includegraphics[width = 0.5\textwidth]{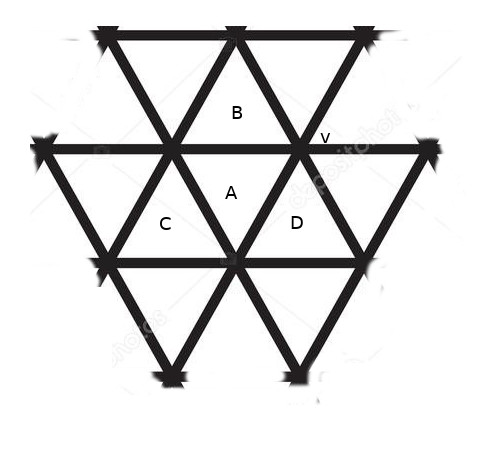}
	\caption{A Locking of Triangle A of Type (3) by Triangles B,C,D of Type (3) on all the Three sides}
	\label{fig:Fourteen}
\end{figure}
The triangle A of type (3) is sorrounded by triangles B,C,D also of type (3). We avoid these this scenario by removing one more vertex $v$ and all the edges incident on it in Figure~\ref{fig:Fourteen}. These edges and faces at the vertex $v$ in Figure~\ref{fig:Fourteen} is similar to a partial pyramid with apex $v$.
Now we remember the sequence in which such vertices $v$ along with the edges on $v$ which are removed. After removing such vertices, among the remaining bands and threads, such a locking scenario does not arise.

So now in the remaining figure, we label, the yellow type of triangles as in snub dodecahedron in Figure~\ref{fig:Thirteen} coherently and then label the blue type of edges and then we label the edges of the threads and those edges where one face containing the edge is already removed, so that, all the faces in the remaining figure are coherently labelled. We can always resort to the method of removing partial pyramids whenever necessary at any stage when the labelling cannot be done.

Then we replace by tracing back the vertices of type $v$ and those edges incident on $v$ which were removed earlier. This addition of $v$ will create faces which we now coherently label just as in Theorem~\ref{theorem:OnePyramid} in a sequential manner. The addition of each such vertex $v$ is just like constructing a partial pyramid where the proof for labelling given in Theorem~\ref{theorem:OnePyramid} can be used.

Now again using Theorem~\ref{theorem:OnePyramid} for the initial removed vertices, we can label the entire simplicial polyhedron coherently. This completes the proof of Theorem~\ref{theorem:SP}.
\end{proof}

\bibliographystyle{abbrv}

\begin{thebibliography}{1}

\bibitem{MR3823190} M.~Aigner, G.~M.~Ziegler,
{\it Proofs from THE BOOK}, Sixth Edition, ISBN-13: 978-3-662-57264-1; 978-3-662-57265-8, Springer, Berlin, 2018, viii+326 pp, \url{https://doi.org/10.1007/978-3-662-57265-8}, MR3823190


\bibitem{MR0077132} F.~Bagemihl,
{\it A Conjecture Concerning Neighbouring Tetrahedra}, 
The American Math Monthly, Vol. {\bf 63}, No. {\bf 5}, May 1956, pp. 328-329, \url{https://www.jstor.org/stable/2310516}, MR0077132

\bibitem{MR0181558} H.~Tietze,
{\it Famous problems of mathematics. Solved and unsolved mathematical problems from antiquity to modern times}, Authorized translation from the second (1959) revised German edition, Edited by B.~K.~Hofstadter and H.~Komm, Graylock Press, New York, 1965, xvi+367 pp, ISBN-13: 978-1124012650, MR0181558

\bibitem{MR0467493} M.~J.~Wenninger,
{\it Polyhedron Models}, Cambridge University Press, London-New York, 1971, xii+208 pp, \url{https://doi.org/10.1017/CBO9780511569746}, MR0467493

\bibitem{MR0730208} M.~J.~Wenninger,
{\it Dual Models}, Cambridge University Press, Cambridge, 1983, xii+156 pp. ISBN: 0-521-24524-9, \url{ https://doi.org/10.1017/CBO9780511569371}, MR0730208 
\end{thebibliography}
\def\cprime{$'$}

\end{document}